\documentclass[11pt]{amsart}
\usepackage{amsmath,amsfonts,amsthm,amsopn,amssymb,latexsym,soul}
\usepackage{array}
\usepackage{cite}
\usepackage{color,enumitem,graphicx}
\usepackage[colorlinks=true,urlcolor=blue,
citecolor=red,linkcolor=blue,linktocpage,pdfpagelabels,
bookmarksnumbered,bookmarksopen]{hyperref}
\usepackage[english]{babel}
\usepackage[left=2.61cm,right=2.61cm,top=2.73cm,bottom=2.73cm]{geometry}

\usepackage{bm}

\usepackage{lineno}
\usepackage{epsfig}
\usepackage{graphics}
\usepackage{float}
\usepackage{multirow}
\usepackage{lineno}
\usepackage{fullpage}
\usepackage[normalem]{ulem} 
\usepackage{xspace}
\usepackage{wrapfig}
\usepackage{color}
\usepackage{indentfirst}
\usepackage{comment}

\theoremstyle{definition}
\newtheorem{theorem}{Theorem}

\newtheorem{prop}{Proposition}
\newtheorem{rmk}{Remark}

\newtheorem{lem}{Lemma}

\numberwithin{defi}{section} 
\numberwithin{prop}{section} 
\numberwithin{theorem}{section}  
\numberwithin{lem}{section}
\numberwithin{equation}{section} 
\numberwithin{rmk}{section}

\allowdisplaybreaks

\title[On the weakest conditions for a Kannan and Chatterjea]{On the weakest conditions for the existence of fixed points of Kannan and Chatterjea type contractions}

\author[S. Hashimoto]{Shunya Hashimoto}
\address{Faculty of Science, Kyoto University,
\newline\indent
Oiwake-tyou Kitashirakawa, Sakyo-ku Kyoto-shi, 606-8502, Japan}
\email{hashimoto.shunya.7m@kyoto-u.ac.jp}

\author[M. Kikkawa]{Misako Kikkawa}
\address{Department of Mathematics, Faculty of Science, Saitama University,
\newline\indent
Shimo-Okubo 255, Sakura-ku Saitama-shi, 338-8570, Japan}
\email{misakokikkawa0731@gmail.com}

\author[S. Machihara]{Shuji Machihara}
\address{Department of Mathematics, Faculty of Science, Saitama University,
\newline\indent
Shimo-Okubo 255, Sakura-ku Saitama-shi, 338-8570, Japan}
\email{machihara@rimath.saitama-u.ac.jp}

\author[A. Saghir]{Aqib Saghir}
\address{Department of Mathematics, Faculty of Science, Saitama University,
\newline\indent
Shimo-Okubo 255, Sakura-ku Saitama-shi, 338-8570, Japan}
\email{saghir.a.460@ms.saitama-u.ac.jp}

\makeatletter
\@namedef{subjclassname@2020}{\textup{2020} Mathematics Subject Classification}
\makeatother

\thanks{
This work was supported by JST CREST, Japan, Grant Number JPMJCR24Q6.
}

\begin{document}
\begin{abstract}
In this paper, we study the weakest possible conditions for fixed point theorems involving two classes of mappings defined by Kannan and Chatterjea. Our approach relies on the so-called CJM condition, which was originally introduced by \'Ciri\'c \cite{Ciric}, and later, Suzuki \cite{Suzuki} showed that the CJM condition is necessary for the existence of fixed points and the convergence of all Picard sequences of Banach type mappings. Our aim is to extend Suzuki's approach to the case of Kannan and Chatterjea mappings. 
In particular, in the first case, we discuss the equivalence of previously known conditions and establish that our conditions are optimal for ensuring that all Picard sequences converge to a fixed point of a mapping.
\end{abstract}

\maketitle

\date{}
\maketitle

\section{Introduction}

Fixed-point theory is an important area in nonlinear analysis, providing foundational results for solving various mathematical and applied problems. Among the classical contractive conditions, Kannan-type mappings have been widely studied due to their unique properties and applications. In 1968, Kannan \cite{Kannan} introduced a contraction condition that does not require strict continuity like Banach \cite{Banach} contractions. 
Throughout this paper, let $X$ be a nonempty set and $(X,d)$ be a metric space.
A mapping \( T: X \to X \) is called a Kannan contraction if there exists a constant 
\( \alpha \in \left[ 0, \frac{1}{2} \right) \) such that:
\[d(Tx, Ty) \leq \alpha \left( d(x, Tx) + d(y, Ty) \right), \text{for all} \ x,y \in X.\]
Several authors \cite{Ariza, Branciari, Huang, Khan, Kuczma, Regan, Reich1, Reich2, Roy} have studied generalizations of Kannan mappings, including integral-type conditions and applications in various spaces. 

Also, Chatterjea \cite{Cha} showed the existence of a fixed point under a contraction mapping of a Chatterjea type with a structure similar to that of a Kannan-type contraction mapping.
A mapping \( T: X \to X \) in a metric space \( (X, d) \) is called a Chatterjea contraction if there exists a constant 
\( \alpha \in \left[ 0, \frac{1}{2} \right) \) such that:
\[d(Tx, Ty) \leq \alpha \left( d(x, Ty) + d(y, Tx) \right), \text{for all} \ x,y \in X.\]

Subrahmanyam \cite{S74} showed that Kannan and Chatterjea type contraction mappings, unlike Banach contraction mappings, give a characterization of the completeness of the metric space.
In other words, it is equivalent that every Kannan-type contraction mapping on the metric space $(X,d)$ has a fixed point and that the metric space is complete. The same is true for Chatterjea contractions.
In this sense, Kannan and Chatterjea type contraction mappings are natural contraction mappings mathematically.

Furthermore, Kannan-type contraction mappings have a convergence condition based on self-difference, so they are expected to have applications in physics, engineering and other areas. 
In fact, Kannan-type contraction mapping was used in \cite{YSP19} to show the existence of solutions of nonlinear partial differential equations describing damped spring-mass systems and deformations of elastic beams. 
And further, graph theory is used and is compatible and works synergistically with Kannan-type contraction mapping. 
We may expect that Kannan-type contraction mapping may also have applications in data science, which deals with network models depending on the distance between inputs and outputs. 
Here we review the development of the weakest conditions especially of Banach type mapping under which admits a fixed point, and then explain what the CJM condition is.

In 1974, L. B. Ćirić \cite{Ciric} introduced the generalization of Banach's contraction principle, which laid the groundwork for what would later become known as CJM contractions. 
Building on this foundation, J. Matkowski \cite{Matkowski} contributed significantly to the development of the CJM framework in 1980 by studying fixed point theorems for contractive mappings in metric spaces, solidifying the theoretical basis for this class of contractions. 
Finally, J. Jachymski \cite{Jachymski} provided a comprehensive study of contractive conditions in 1995, establishing equivalent conditions for Meir-Keeler type contractions \cite{Meir} and further advancing the understanding of generalized contractions. 
CJM contractions provide a powerful and unified framework for analyzing fixed point theorems in metric spaces. 
Specifically, a mapping \( T \) on a metric space \( (X, d) \) satisfies the CJM condition for Banach type mapping if the following holds for all \( x, y \in X \):
\begin{enumerate}
\item[(i)] For every $\varepsilon>0$, there exists $\delta>0$ such that $d(x, y)<\varepsilon+\delta$ implies $d(Tx,Ty)\le \varepsilon$.
\item[(ii)] $x\not=y$ implies $d(Tx, Ty)<d(x, y)$.
\end{enumerate}
The CJM condition in the Banach-type fixed point theorem is nearly the weakest possible. 

A modified version of this condition, which is the weakest, was introduced by Suzuki \cite{Suzuki} (see Remark \ref{rmk3.1}) who showed that this weakest condition is equivalent to the existence of a fixed point of a mapping and the convergence of all its Picard sequences to the fixed point. The contribution of this paper is to extend the result of Suzuki for Kannan and Chatterjea type contractions.

\section{The CJM version of a Kannan-type contraction mapping}

The following fixed point theorem is proved in preparation for stating the weakest condition.
\begin{theorem}
\label{2.1}
Let $(X,d)$ be a complete metric space and suppose that the mapping $T: X\to X$ satisfies the following.
\begin{enumerate}
\item[(i)] For every $\varepsilon>0$, there exists $\delta>0$ such that $\frac{1}{2}\{d(x,Tx)+d(y,Ty)\}<\varepsilon+\delta$ implies $d(Tx,Ty)\le \varepsilon$.
\item[(ii)] $x\not=y$ implies $d(Tx,Ty)<\frac{1}{2}\{d(x,Tx)+d(y,Ty)\}$.
\end{enumerate}
Then, $T$ has a unique fixed point.
\end{theorem}
\begin{proof}
Fix $x_0\in X$. Set $x_n=T^nx_0$ for all $n\in\mathbb{N}$. 
We may follow Theorem 2 in \cite{Jachymski} to have that $(x_n)_{n\in\mathbb{N}}$ is a Cauchy sequence, and we omit it here. And then, since $(X,d)$ is complete, there exists $u\in X$ such that $x_n\to u$ as $n\to\infty$. From
\begin{align*}
d(u,Tu)&\le d(u,x_n)+d(x_n,Tx_n)+d(Tx_n,Tu) \\
&<d(u,x_n)+d(x_n,Tx_n)+\frac{1}{2}\{d(x_n,Tx_n)+d(u,Tu)\},
\end{align*}
we have
\[ \frac{1}{2}d(u,Tu)<d(u,x_n)+\frac{3}{2}d(x_n,Tx_n)\to 0 \ (n\to \infty). \]
Then we get $u=Tu$.

Finally, we show uniqueness. Let $u\not=v, \ Tu=u$ and $Tv=v$. Then we have
\[ d(u,v)=d(Tu,Tv)<\frac{1}{2}\{d(u,Tu)+d(v,Tv)\}=0, \]
which is a contradiction. This completes the proof of the theorem.
\end{proof}

\section{The weakest condition for a Kannan-type contraction mapping}

The weakest condition is a slightly weakened version of the condition for contraction mapping in Section 2.
\begin{theorem}
\label{3.1}
Let $(X,d)$ be a complete metric space and suppose that the mapping $T: X\to X$ satisfies the following.
\begin{align}
\label{CM}
x\not=y \ \text{implies} \ d(Tx,Ty)<\frac{1}{2}\{d(x,Tx)+d(y,Ty)\}.
\tag{CM}
\end{align}
Then, the following are equivalent.
\begin{enumerate}
\item[(i)] For any $x\in X$ and $\varepsilon>0$, there exists $\delta>0$ such that \\  
$\frac{1}{2}\{d(T^ix,T^{i+1}x)+d(T^jx,T^{j+1}x)\}<\varepsilon+\delta$ implies $d(T^{i+1}x,T^{j+1}x)\le \varepsilon$ for all $i,j\in\mathbb{N}\cup\{0\}$.
\item[(ii)] $T$ has a unique fixed point $z\in X$ and $(T^nx)$ converges to $z$ for any $x\in X$.
\end{enumerate}
\end{theorem}

\begin{rmk}
\label{rmk3.1}
Condition (i) of Theorem \ref{2.1} is an assumption for any $x,y\in X$.
On the other hand, condition (i) of Theorem \ref{3.1} is an assumption restricted to the Picard sequence $(T^nx)$.
Therefore, condition (i) of Theorem \ref{3.1} is weaker than condition (i) of Theorem \ref{2.1}.
This is based on the weakest condition for Banach's contraction mapping by Suzuki \cite{Suzuki}.
\end{rmk}

\begin{proof}[Proof of Theorem \ref{3.1}]
((i)$\Rightarrow$(ii)) Fix $x_0\in X$. Set $x_n=T^nx_0$ for all $n\in\mathbb{N}$. 
If there exists $N\in \mathbb{N}$ such that $x_N=x_{N+1}$, $x_N$ is a fixed point of $T$. Therefore, assume that $x_n\not=x_{n+1}$ for any $n\in\mathbb{N}$. Then, we have
\begin{align*}
d(x_n,x_{n+1})&=d(Tx_{n-1},Tx_n)<\frac{1}{2}\{d(x_{n-1},Tx_{n-1})+d(x_n,Tx_n)\} \\
&=\frac{1}{2}\{d(x_{n-1},x_n)+d(x_n,x_{n+1})\}.
\end{align*}
So, 
\[ \frac{1}{2}d(x_n,x_{n+1})<\frac{1}{2}d(x_{n-1},x_n). \]
Therefore, $(d(x_n,x_{n+1}))_{n\in\mathbb{N}}$ is a strictly decreasing sequence.

Thus, there exists $\varepsilon_0\ge0$ such that $d(x_n,x_{n+1})\to \varepsilon_0$ as $n\to\infty$ and $d(x_n,x_{n+1})>\varepsilon_0$ for any $n$.

Assume that $\varepsilon_0>0$. Then there exist $\delta_0>0$ and $k\in\mathbb{N}$ such that
\[ \frac{1}{2}\{d(x_k,Tx_k)+d(x_{k+1},Tx_{k+1})\}<\varepsilon_0+\delta_0. \]
By assumption (i) we have $d(Tx_k,Tx_{k+1})\le \varepsilon_0$, which is a contradiction. Therefore, $\varepsilon_0=0$ and 
\begin{align*}
\lim_{n\to \infty}d(x_n,x_{n+1})=0.
\end{align*}
Next, we show that $(x_n)$ is a Cauchy sequence. 
For this, first note that if $x=y$, then we always have
\[ d(Tx,Ty)\le \frac{1}{2}\{d(x,Tx)+d(y,Ty)\}. \]
For any $n,m\in\mathbb{N}$ with $n,m\ge 2$, we have
\begin{align*}
d(x_n,x_m)&=d(Tx_{n-1},Tx_{m-1}) \\
&\le \frac{1}{2}\{d(x_{n-1},Tx_{n-1})+d(x_{m-1},Tx_{m-1})\} \\
&=\frac{1}{2}\{d(x_{n-1},x_n)+d(x_{m-1},x_m)\}\to 0 \ \text{as} \ n,m\to \infty.
\end{align*}
Therefore, $(x_n)_{n\in\mathbb{N}}$ is a Cauchy sequence, and the remainder of the proof follows as in Theorem \ref{2.1}.

((ii)$\Rightarrow$(i)) Let $z\in X$ be a fixed point of $T$. Arguing by contradiction, we assume that (i) does not hold, that is, there exist $x\in X$ and $\varepsilon>0$ such that for all $n\in \mathbb{N}$ there exists $(f(n)),(g(n))\subset \mathbb{N}\cup\{0\}$ such that
\begin{align}
&f(n)<g(n), \nonumber \\
&\frac{1}{2}\{d(x_{f(n)},x_{f(n)+1})+d(x_{g(n)},x_{g(n)+1})\}<\varepsilon+\frac{1}{n} \quad \text{and} \\
\label{32}
&\varepsilon<d(x_{f(n)+1},x_{g(n)+1}),
\end{align}
where $x_{f(n)}=T^{f(n)}x$. Without loss of generality, we may assume that either
\begin{itemize}
\item $\displaystyle \lim_{n\to \infty}f(n)=\infty$ \ or
\item There exists $\mu\in \mathbb{N}\cup\{0\}$ such that $f(n)=\mu$ for all $n\in\mathbb{N}$.
\end{itemize}
Note that if $(f(n))$ oscillates, the second case is attributed. Similarly, we may assume that either
\begin{itemize}
\item $\displaystyle \lim_{n\to \infty}g(n)=\infty$ \ or
\item There exists $\nu\in \mathbb{N}\cup\{0\}$ such that $g(n)=\nu$ for all $n\in\mathbb{N}$.
\end{itemize}
By \eqref{32}, we have $x_{f(n)}\not=x_{g(n)}$. So by \eqref{CM}, we obtain
\begin{align*}
\varepsilon&<d(x_{f(n)+1},x_{g(n)+1})<\frac{1}{2}\{d(x_{f(n)},x_{f(n)+1})+d(x_{g(n)},x_{g(n)+1})\}<\varepsilon+\frac{1}{n}.
\end{align*}
Therefore, we get
\begin{align}
\label{33}
0<\varepsilon&=\lim_{n\to\infty}d(x_{f(n)+1},x_{g(n)+1}) \nonumber \\
&=\frac{1}{2}\lim_{n\to\infty}\{d(x_{f(n)},x_{f(n)+1})+d(x_{g(n)},x_{g(n)+1})\}.
\end{align}
Since $(x_n)$ converges to $z$, $f(n)\not\to \infty$ or $g(n)\not\to \infty$ holds. Suppose that $f(n)\not\to \infty$ holds, that is, $f(n)=\mu$ for all $n\in\mathbb{N}$ is satisfied.

First, if $g(n)=\nu$ for all $n\in\mathbb{N}$, we have
\[ d(x_{\mu+1},x_{\nu+1})=\varepsilon, \]
which is a contradiction. 
On the other hand, if $\displaystyle \lim_{n\to\infty}g(n)=\infty$, then $x_{g(n)}\to z$, so from \eqref{33} we have
\begin{align}
\label{34}
d(x_{\mu+1},z)=\varepsilon&=\frac{1}{2}d(x_{\mu},x_{\mu+1})\le \frac{1}{2}\{d(x_{\mu},z)+d(z,x_{\mu+1})\}. 
\end{align}
We can show the first equality in \eqref{34} since we obtain
\begin{align*}
|d(x_{\mu+1},x_{g(n)+1})-d(x_{\mu+1},z)|\le d(x_{g(n)+1},z)\to 0, \ \text{as} \ n\to \infty,
\end{align*}
from triangle inequality 
\begin{align*}
d(x_{\mu+1},x_{g(n)+1})-d(x_{\mu+1},z)&\le d(x_{g(n)+1},z), \\
d(x_{\mu+1},z)-d(x_{\mu+1},x_{g(n)+1})&\le d(x_{g(n)+1},z).
\end{align*}
From the first equality in \eqref{34}, $d(x_{\mu+1},z)\not=0$ holds. By \eqref{34}, we get
\[ \frac{1}{2}d(x_{\mu+1},z)\le \frac{1}{2}d(x_{\mu},z). \]
Then, $d(x_{\mu},z)\not=0$ holds. So by \eqref{CM}, we obtain
\begin{align*}
d(x_{\mu+1},z)=d(Tx_{\mu},Tz)<\frac{1}{2}\{d(x_{\mu},x_{\mu+1})+d(z,Tz)\}=\frac{1}{2}d(x_{\mu},x_{\mu+1}),
\end{align*}
which is a contradiction with \eqref{34}. Therefore, (i) holds.
\end{proof}

\section{The case of Chatterjea type contractions}

Similarly, we can show a fixed point theorem and weakest condition in the CJM version for contraction mappings of Chatterjea type.
\begin{theorem}
\label{4.1}
Let $(X,d)$ be a complete metric space and suppose that the mapping $T: X\to X$ satisfies the following.
\begin{enumerate}
\item[(i)] For every $\varepsilon>0$, there exists $\delta>0$ such that $\frac{1}{2}\{d(x,Ty)+d(y,Tx)\}<\varepsilon+\delta$ implies $d(Tx,Ty)\le \varepsilon$.
\item[(ii)] $x\not=y$ implies $d(Tx,Ty)<\frac{1}{2}\{d(x,Ty)+d(y,Tx)\}$.
\end{enumerate}
Then, $T$ has a unique fixed point.
\end{theorem}
\begin{proof}
Fix $x_0\in X$. Set $x_n=T^nx_0$ for all $n\in\mathbb{N}$. 
We may follow Theorem 2 in \cite{Jachymski} to have that $(x_n)_{n\in\mathbb{N}}$ is a Cauchy sequence, and we omit it here. And then, since $(X,d)$ is complete, there exists $u\in X$ such that $x_n\to u$ as $n\to\infty$. From
\begin{align*}
&d(u,Tu) \\
&\le d(u,x_n)+d(x_n,Tx_n)+d(Tx_n,Tu) \\
&<d(u,x_n)+d(x_n,Tx_n)+\frac{1}{2}\{d(x_n,Tu)+d(u,Tx_n)\} \\
&\le d(u,x_n)+d(x_n,Tx_n)+\frac{1}{2}d(x_n,Tu)+\frac{1}{2}d(u,x_n)+\frac{1}{2}d(x_n,Tx_n) \\
&\le d(u,x_n)+d(x_n,Tx_n)+\frac{1}{2}\{d(x_n,u)+d(u,Tu)\} \\
& \quad \quad +\frac{1}{2}d(u,x_n)+\frac{1}{2}d(x_n,Tx_n),
\end{align*}
we have
\[ \frac{1}{2}d(u,Tu)<2d(u,x_n)+\frac{3}{2}d(x_n,Tx_n)\to 0 \ (n\to \infty). \]
Then we get $u=Tu$.

For uniqueness, let $u\not=v, \ Tu=u$ and $Tv=v$. Then we have
\begin{align*}
d(u,v)&=d(Tu,Tv)<\frac{1}{2}\{d(u,Tv)+d(v,Tu)\} \\
&\le \frac{1}{2}\{d(u,v)+d(v,Tv)\}+\frac{1}{2}\{d(v,u)+d(u,Tu)\}=d(u,v),
\end{align*}
which is a contradiction.
\end{proof}
\begin{theorem}
\label{4.2}
Let $(X,d)$ be a complete metric space and suppose that the mapping $T: X\to X$ satisfies the following:
\begin{align}
\label{CM2}
x\not=y \ \text{implies} \ d(Tx,Ty)<\frac{1}{2}\{d(x,Ty)+d(y,Tx)\}.
\tag{CM2}
\end{align}
Then, the following are equivalent.
\begin{enumerate}
\item[(i)] For any $x\in X$ and $\varepsilon>0$, there exists $\delta>0$ such that \\  
$\frac{1}{2}\{d(T^ix,T^{j+1}x)+d(T^jx,T^{i+1}x)\}<\varepsilon+\delta$ implies $d(T^{i+1}x,T^{j+1}x)\le \varepsilon$ for all $i,j\in\mathbb{N}\cup\{0\}$.
\item[(ii)] $T$ has a unique fixed point $z\in X$ and $(T^nx)$ converges to $z$ for any $x\in X$.
\end{enumerate}
\end{theorem}
\begin{proof}
((i)$\Rightarrow$(ii)) Fix $x_0\in X$. Set $x_n=T^nx_0$ for all $n\in\mathbb{N}$.
If there exists $N\in \mathbb{N}$ such that $x_N=x_{N+1}$, $x_N$ is a fixed point of $T$. Therefore, assume that $x_n\not=x_{n+1}$ for any $n\in\mathbb{N}$. Then, we have
\begin{align*}
d(x_n,x_{n+1})&=d(Tx_{n-1},Tx_n)<\frac{1}{2}\{d(x_{n-1},Tx_n)+d(x_n,Tx_{n-1})\} \\
&=\frac{1}{2}\{d(x_{n-1},x_{n+1})+d(x_n,x_n)\}\le \frac{1}{2}\{d(x_{n-1},x_n)+d(x_n,x_{n+1})\}.
\end{align*}
So, 
\[ \frac{1}{2}d(x_n,x_{n+1})<\frac{1}{2}d(x_{n-1},x_n). \]
Therefore, $(d(x_n,x_{n+1}))_{n\in\mathbb{N}}$ is a strictly decreasing sequence.

Thus, there exists $\varepsilon_0\ge0$ such that $d(x_n,x_{n+1})\to \varepsilon_0$ as $n\to\infty$ and $d(x_n,x_{n+1})>\varepsilon_0$ for any $n$.

Assume that $\varepsilon_0>0$. Then there exist $\delta_0>0$ and $k\in\mathbb{N}$ such that
\[ \frac{1}{2}\{d(x_k,Tx_{k+1})+d(x_{k+1},Tx_k)\}\le \frac{1}{2}\{d(x_k,x_{k+1})+d(x_{k+1},Tx_{k+1})\}<\varepsilon_0+\delta_0. \]
By assumption (i) we have $d(Tx_k,Tx_{k+1})\le \varepsilon_0$, which is a contradiction. Therefore, $\varepsilon_0=0$ and 
\begin{align}
\label{41}
\lim_{n\to \infty}d(x_n,x_{n+1})=0.
\end{align}
Next, we show that $(x_n)$ is a Cauchy sequence. Let $\varepsilon>0$ be arbitrary and let $\delta$ be chosen accordingly as (i).
By \eqref{41}, there exists some $k\in\mathbb{N}$ such that $d(x_{n-1},x_n)<\delta$ for each $n>k$. Fix $n>k$. It suffices to show that
\begin{align*}
d(x_n,x_{n+p})\le \varepsilon, \ \text{for all} \ p\in\mathbb{N}.
\end{align*}
We show by induction.
The case of $p=1$ is obvious.
We show the case of $p$ assuming the case of $p-1$. In this case, we have
\begin{align*}
&\frac{1}{2}\{ d(x_{n-1},Tx_{n+p-1})+d(x_{n+p-1},Tx_{n-1})\} \\
&=\frac{1}{2}\{ d(x_{n-1},x_{n+p})+d(x_{n+p-1},x_{n})\} \\
&\le \frac{1}{2}\{ d(x_{n-1},x_n)+d(x_n,x_{n+p-1})+d(x_{n+p-1},x_{n+p})+d(x_{n+p-1},x_n)\} \\
&<\delta+\varepsilon.
\end{align*}
Therefore,
\[ d(x_n,x_{n+p})=d(Tx_{n-1},Tx_{n+p-1})\le \varepsilon. \]
Therefore, $(x_n)_{n\in\mathbb{N}}$ is a Cauchy sequence, and the remainder of the proof follows as in Theorem \ref{4.1}.

((ii)$\Rightarrow$(i)) Let $z\in X$ be a fixed point of $T$. Arguing by contradiction, we assume that (i) does not hold, that is, there exist $x\in X$ and $\varepsilon>0$ such that for all $n\in \mathbb{N}$ there exists $(f(n)),(g(n))\subset \mathbb{N}\cup\{0\}$ such that
\begin{align}
&f(n)<g(n), \nonumber \\
&\frac{1}{2}\{d(x_{f(n)},x_{g(n)+1})+d(x_{g(n)},x_{f(n)+1})\}<\varepsilon+\frac{1}{n} \quad \text{and} \\
\label{43}
&\varepsilon<d(x_{f(n)+1},x_{g(n)+1}),
\end{align}
where $x_{f(n)}=T^{f(n)}x$. Without loss of generality, we may assume that either
\begin{itemize}
\item $\displaystyle \lim_{n\to \infty}f(n)=\infty$ \ or
\item There exists $\mu\in \mathbb{N}\cup\{0\}$ such that $f(n)=\mu$ for all $n\in\mathbb{N}$.
\end{itemize}
Note that if $(f(n))$ oscillates, the second case is attributed. Similarly, we may assume that either
\begin{itemize}
\item $\displaystyle \lim_{n\to \infty}g(n)=\infty$ \ or
\item There exists $\nu\in \mathbb{N}\cup\{0\}$ such that $g(n)=\nu$ for all $n\in\mathbb{N}$.
\end{itemize}
By \eqref{43}, we have $x_{f(n)}\not=x_{g(n)}$. So by \eqref{CM2}, we obtain
\begin{align*}
\varepsilon&<d(x_{f(n)+1},x_{g(n)+1})<\frac{1}{2}\{d(x_{f(n)},x_{g(n)+1})+d(x_{g(n)},x_{f(n)+1})\}<\varepsilon+\frac{1}{n}.
\end{align*}
Therefore, we get
\begin{align}
\label{44}
0<\varepsilon&=\lim_{n\to\infty}d(x_{f(n)+1},x_{g(n)+1}) \nonumber \\
&=\frac{1}{2}\lim_{n\to\infty}\{d(x_{f(n)},x_{g(n)+1})+d(x_{g(n)},x_{f(n)+1})\}.
\end{align}
Since $(x_n)$ converges to $z$, $f(n)\not\to \infty$ or $g(n)\not\to \infty$ holds. Suppose that $f(n)\not\to \infty$ holds, that is, $f(n)=\mu$ for all $n\in\mathbb{N}$ is satisfied.

First, if $g(n)=\nu$ for all $n\in\mathbb{N}$, we have
\[ d(x_{\mu+1},x_{\nu+1})=\varepsilon, \]
which is a contradiction. 
On the other hand, if $\displaystyle \lim_{n\to\infty}g(n)=\infty$, then $x_{g(n)}\to z$, so from \eqref{44} we have
\begin{align}
\label{45}
d(x_{\mu+1},z)=\varepsilon&=\frac{1}{2}\lim_{n\to\infty}\{d(x_{f(n)},x_{g(n)+1})+d(x_{g(n)},x_{f(n)+1})\} \nonumber \\
&=\frac{1}{2}\{d(x_{\mu},z)+d(z,x_{\mu+1})\}. 
\end{align}
We can show the third equality in \eqref{45} since we obtain
\begin{align*}
&|\{d(x_{\mu},z)+d(z,x_{\mu+1})\}-\{d(x_{\mu},x_{g(n)+1})+d(x_{g(n)},x_{\mu+1})\}| \\
&\le |d(x_{g(n)},z)|+|d(x_{g(n)+1},z)|\to 0, \ \text{as} \ n\to \infty,
\end{align*}
from the triangle inequality 
\begin{align*}
d(x_{\mu+1},z)-d(x_{\mu+1},x_{g(k)})&\le d(x_{g(k)},z), \\
d(x_{\mu},z)-d(x_{\mu},x_{g(k)+1})&\le d(x_{g(k)+1},z).
\end{align*}
By \eqref{45}, $d(x_{\mu},z)\not=0$ holds. So by \eqref{CM2}, we obtain
\begin{align*}
d(x_{\mu+1},z)&=d(Tx_{\mu},Tz)<\frac{1}{2}\{d(x_{\mu},Tz)+d(z,Tx_{\mu})\} \\
&=\frac{1}{2}\{d(x_{\mu},z)+d(z,x_{\mu+1})\},
\end{align*}
which is a contradiction with \eqref{45}. Therefore, (i) holds.
\end{proof}

\section{Discussion on the weakest conditions}

In this section, we discuss the equivalence of the weakest conditions.
As shown in the proof of Theorem \ref{2.1}, under the \eqref{CM} condition, the sequence $(d(T^nx,T^{n+1}x))_{n\in\mathbb{N}}$ is strictly decreasing, and hence its nonnegative limit exists.
Let us denote this limit by $\alpha$.
\begin{prop}
\label{5.1}
Assume the \eqref{CM} condition of Theorem \ref{3.1}. 
Let $\alpha$ be the limit of the sequence $(d(T^nx,T^{n+1}x))_{n\in\mathbb{N}}$. Then the following statements are equivalent.
\begin{enumerate}
\item[(i)] $\alpha=0.$
\item[(ii)] For any $\varepsilon>0$, there exists $\delta>0$ such that 

\noindent
$\frac{1}{2}\{ d(T^ix,T^{i+1}x)+d(T^{i+1}x,T^{i+2}x)\}<\varepsilon+\delta$ implies $d(T^{i+1}x,T^{i+2}x)\le \varepsilon$ for all $i\in\mathbb{N}$.
\item[(iii)] For any $\varepsilon>0$, there exists $\delta>0$ such that 

\noindent
$\frac{1}{2}\{ d(T^ix,T^{i+1}x)+d(T^{i+1}x,T^{i+2}x)\}<\varepsilon+\delta$ implies $d(T^{i+2}x,T^{i+3}x)\le \varepsilon$ for all $i\in\mathbb{N}$.
\item[(iv)] For any $\varepsilon\in (0,\infty)\backslash \{d(T^kx,T^{k+1}x) : k=1,2,\cdots\}$, there exists $\delta>0$ such that \\
$\frac{1}{2}\{ d(T^ix,T^{i+1}x)+d(T^{j}x,T^{j+1}x)\}<\varepsilon+\delta$ implies $d(T^{i+1}x,T^{j+1}x)\le \varepsilon$ for all $i,j\in\mathbb{N}$.
\end{enumerate}
Moreover, if $(X,d)$ is complete, then the following is also equivalent to the above.
\begin{enumerate}
\item[(v)] For any $\varepsilon>0$, there exists $\delta>0$ such that 

\noindent
$\frac{1}{2}\{ d(T^ix,T^{i+1}x)+d(T^{j}x,T^{j+1}x)\}<\varepsilon+\delta$ implies $d(T^{i+1}x,T^{j+1}x)\le \varepsilon$ for all $i,j\in\mathbb{N}$.
\end{enumerate}
\end{prop}
This proposition follows from the following lemma.
\begin{lem}
\label{lem1}
Let $(a_n)_{n\in\mathbb{N}}$ be a sequence satisfying
\[ a_1>a_2>\cdots>a_n>a_{n+1}>\cdots\ge 0. \]
And let $\alpha\ge 0$ be the limit of $a_n$.
Then, the following (i)-(iv) are equivalent, and (v) is strictly stronger than them.
\begin{enumerate}
\item[(i)] $\alpha=0.$
\item[(ii)] For any $k\ge0$ and $\varepsilon>0$, there exists $\delta>0$ such that $a_n<\varepsilon+\delta$ implies $a_{n+k}\le \varepsilon$ for all $n\in\mathbb{N}$.
\item[(iii)] For any $k\ge1$ and $\varepsilon>0$, there exists $\delta>0$ such that $\frac{1}{2}(a_n+a_{n+1})<\varepsilon+\delta$ implies $a_{n+k}\le \varepsilon$ for all $n\in\mathbb{N}$.
\item[(iv)] For any $\varepsilon\in (0,\infty)\backslash \{a_k : k=1,2,\cdots\}$, there exists $\delta>0$ such that $a_m+a_n<\varepsilon+\delta$ implies $a_m+a_n\le \varepsilon$ for all $m,n\in\mathbb{N}$.
\item[(v)] For any $\varepsilon>0$, there exists $\delta>0$ such that $a_m+a_n<\varepsilon+\delta$ implies $a_m+a_n\le \varepsilon$ for all $m,n\in\mathbb{N}$.
\end{enumerate}
\end{lem}
\begin{proof}
We denote (ii) for each $k\ge 0$ as (ii)$_k$.

((i)$\Rightarrow$(ii)) We show (i)$\Rightarrow$(ii)$_0$, since (ii)$_k \Rightarrow $(ii)$_{k+1}$ is obvious. 
In the case of $\varepsilon\ge a_1$, 
we have $a_{n+k}\le \varepsilon$ for all $n\in\mathbb{N}$, then (ii) holds. In the case of $\varepsilon\in (0,a_1)$, there exists $n_0\in\mathbb{N}$ such that 
\[ \cdots>a_{n_0-1}>\varepsilon\ge a_{n_0}>a_{n_0+1}>\cdots. \]
Therefore, we can take $\delta>0$ to satisfy 
\[ \cdots>a_{n_0-1}>\varepsilon+\delta>\varepsilon\ge a_{n_0}>a_{n_0+1}>\cdots, \]
so (ii)$_0$ holds.

((ii)$\Rightarrow$(i)) For fixed $k\ge 0$, we show that (ii)$_k \Rightarrow $(i). Assume $\alpha>0$.
We take $\varepsilon=\alpha$, then there exists $\delta>0$ of (ii)$_k$.
From $a_n\downarrow \varepsilon$ by assumption, $a_n>\varepsilon$ holds for any $n\in\mathbb{N}$, and there exists $n_0\in\mathbb{N}$ such that $a_{n_0}<\varepsilon+\delta$.
Therefore, from (ii)$_k$, we have $a_{n_0+k}\le \varepsilon$, which is a contradiction.

((ii)$\Leftrightarrow$(iii)) It is obvious from the following.
\[ \cdots>a_{n-1}>\frac{a_{n-1}+a_n}{2}>a_n>\frac{a_n+a_{n+1}}{2}>a_{n+1}>\cdots. \]

((i)$\Rightarrow$(iv)) Let $\varepsilon\in (0,\infty)\backslash \{a_k : k=1,2,\cdots\}$.
Without loss of generality, we assume $n<m$.
By (i), there exists $n_0\in\mathbb{N}$ such that 
\[ \varepsilon>a_{n_0}>a_{n_0+1}>\cdots. \]
We take $\delta_1>0$ such that $\varepsilon>a_{n_0}+\delta_1$.
By (i), there exists $n_1>n_0$ such that 
\[ 0<a_m<\delta_1 \]
holds for all $m\ge n_1$.
Then for all $n_0\le n<n_1\le m$, we have 
\[ a_m+a_n\le \delta_1+a_{n_0}<\varepsilon. \]
We note that there are only a finite number of pairs $(n,m)\in \mathbb{N}^2$ satisfying $n_0\le n<m<n_1$. Therefore, we define $\delta>0$ as follows.
\begin{align*}
\delta&=\min\{\delta_1, \min_{m,n}\delta_{m,n}\}, \\
\delta_{m,n}&=
\begin{cases}
a_n+a_m-\varepsilon, & (a_n+a_m>\varepsilon) \\
a_1, & (a_n+a_m\le \varepsilon).
\end{cases}
\end{align*}
Then if $a_m+a_n<\varepsilon+\delta$, we have $a_m+a_n\le \varepsilon$. So (iv) holds.

((iv)$\Rightarrow$(i)) Assume $\alpha>0$.
We use (iv) as $m=n$, that is, for any $\varepsilon\in (0,\infty)\backslash \{\frac{a_k}{2} : k=1,2,\cdots\}$, there exists $\delta>0$ such that $a_n<\varepsilon+\delta$ implies $a_n\le \varepsilon$ for all $n\in\mathbb{N}$.

Case 1: If $\alpha\in (0,\infty)\backslash \{\frac{a_k}{2} : k=1,2,\cdots\}$, we can take $\varepsilon=\alpha$, then there exists $\delta>0$ of (iv).
From $a_n\downarrow \varepsilon$ by assumption, $a_n>\varepsilon$ holds for any $n\in\mathbb{N}$, and there exists $n_0\in\mathbb{N}$ such that $a_{n_0}<\varepsilon+\delta$.
Therefore, from (iv), we have $a_{n_0+1}\le \varepsilon$, which is a contradiction. 

Case 2: If $\alpha=a_{k_0}$ for some $k_0$, we have $a_{k_0+1}<\alpha$, which contradicts $a_n\downarrow \alpha$.

((v) is strictly stronger than (i)) Since (v)$\Rightarrow$(iv) is obvious, (v)$\Rightarrow $(i) holds.
However, (i)$\Rightarrow$(v) does not hold.
In fact, if we take $\varepsilon=a_{n_0}$, then for any $\delta>0$ there exists $0<a_m<\delta$ from (i). 
This implies $\varepsilon<a_{n_0}+a_m<\varepsilon+\delta$. So (v) does not hold.
\end{proof}
\begin{proof}[Proof of Proposition \ref{5.1}]
By setting $a_n=d(T^nx,T^{n+1}x)$, it follows from Lemma \ref{lem1} that conditions (i)-(iv) in Proposition \ref{5.1} are equivalent. 
If $(X,d)$ is complete, (v) is also equivalent by Theorem \ref{3.1}.
\end{proof}
From Proposition \ref{5.1}, we can rewrite Theorem \ref{3.1} as follows.
\begin{theorem}
\label{5.2}
Let $(X,d)$ be a complete metric space and suppose that the mapping $T: X\to X$ satisfies the following:
\begin{align}
x\not=y \ \text{implies} \ d(Tx,Ty)<\frac{1}{2}\{d(x,Tx)+d(y,Ty)\}.
\tag{CM}
\end{align}
Then, the following are equivalent.
\begin{enumerate}
\item[(i)] $\displaystyle \lim_{n\to\infty}d(T^nx,T^{n+1}x)=0$ for any $x\in X$.
\item[(ii)] $T$ has a unique fixed point $z\in X$ and $\{T^nx\}$ converges to $z$ for any $x\in X$.
\end{enumerate}
\end{theorem}

\section*{Declarations}

\subsection*{Conflicts of interests}
The authors declare that there is no conflict of interest regarding the publication of this paper.

\subsection*{Data Availability Statements}
Data sharing is not applicable to this article as no datasets were generated or analyzed during the current study.

\end{document}